\documentclass[a4paper,11pt]{amsart}
\usepackage{amsmath}
\usepackage{amssymb}
\usepackage{amsthm}
\usepackage{graphics}
\usepackage[abbrev]{amsrefs}

\usepackage{amsthm}
\usepackage[all,cmtip]{xy}

\newtheorem{thm}{Theorem}[section]
\newtheorem{lem}[thm]{Lemma}

\newtheorem{cor}[thm]{Corollary}
\newtheorem{prop}[thm]{Proposition}

\newtheorem{claim}[thm]{Claim}
\newtheorem{conj}[thm]{Conjecture}

\theoremstyle{definition}
\newtheorem{defi}[thm]{Definition}

\theoremstyle{remark}
\newtheorem{rmk}[thm]{Remark}

\newcommand{\mbQ}{\mathbb{Q}}

\newcommand{\mbR}{\mathbb{R}}
\newcommand{\mbC}{\mathbb{C}}

\newcommand{\mbZ}{\mathbb{Z}}

\newcommand{\mcD}{\mathcal{D}}

\newcommand{\mcO}{\mathcal{O}}

\newcommand{\mfa}{\mathfrak{a}}
\newcommand{\mfb}{\mathfrak{b}}
\newcommand{\mfc}{\mathfrak{c}}

\newcommand{\mfA}{\mathfrak{A}}

\newcommand{\mfG}{\mathfrak{G}}
\newcommand{\mfL}{\mathfrak{L}}

\DeclareMathOperator{\mld}{mld}

\DeclareMathOperator{\ord}{ord}

\DeclareMathOperator{\Supp}{Supp}

\DeclareMathOperator{\Span}{Span}

\DeclareMathOperator{\cent}{c}

\DeclareMathOperator{\vol}{vol}
\DeclareMathOperator{\coeff}{coeff}
\DeclareMathOperator{\divi}{div}
\DeclareMathOperator{\LCT}{LCT}
\DeclareMathOperator{\Diff}{Diff}

\usepackage{framed}

\makeatletter
 {\endMakeFramed}
\makeatother

\title[mld's on varieties with fixed Gorenstein index]
{On minimal log discrepancies on varieties with fixed Gorenstein index}

\author{Yusuke Nakamura}
\address{Graduate School of Mathematical Sciences, 
the University of Tokyo, 3-8-1 Komaba, Meguro-ku, Tokyo 153-8914, Japan.}
\email{nakamura@ms.u-tokyo.ac.jp}

\begin{document}
\begin{abstract}
We generalize the rationality theorem of 
the accumulation points of log canonical thresholds 
which was proved by Hacon, M\textsuperscript{c}Kernan, and Xu. 
Further, we apply the rationality to the ACC problem on the minimal log discrepancies. 
We study the set of log discrepancies on varieties with fixed Gorenstein index. 
As a corollary, we prove that the 
minimal log discrepancies of three-dimensional canonical pairs with fixed coefficients satisfy the ACC. 
\end{abstract}

\maketitle

\section{Introduction}
The minimal log discrepancy (mld for short) was introduced by Shokurov, in order to reduce the conjecture of 
terminations of flips to a local problem about singularities. 
Recently, this has been a fundamental invariant in the minimal model program. 
There are two conjectures on mld's, 
the ACC (ascending chain condition) conjecture and the LSC (lower semi-continuity) conjecture. 
Shokurov showed that these two conjectures imply the conjecture of terminations of flips \cite{Shokurov:letter}. 

In this paper, we consider the ACC conjecture. 
For an $\mbR$-divisor $D$ and a subset $I \subset \mbR$, 
we write $D \in I$ when all the non-zero coefficients of $D$ belong to $I$. 
Further, for a subset $I \subset \mbR$, we say that $I$ satisfies 
\textit{the ascending chain condition} (resp.\ \textit{the descending chain condition}) 
when there is no infinite increasing (resp.\ decreasing) sequence $a_i \in I$. 
\textit{ACC} (resp.\ \textit{DCC}) stands for the ascending chain condition (resp.\ the descending chain condition). 

\begin{conj}[ACC conjecture {\cite[Conjecture 4.2]{Shokurov:models}}]\label{conj:ACC}
Fix $d \in \mbZ _{>0}$ and a subset $I \subset [0, 1]$ which satisfies the DCC. 
Then the following set 
\[
A(d,I) := \{ \mld _x (X, \Delta) \mid \text{$(X, \Delta)$ is a log pair, $\dim X = d$, $\Delta \in I$, $x \in X$}\}
\]
satisfies the ACC, where $x$ is a closed point of $X$. 
\end{conj}
We are mainly interested in the case when $I$ is a finite set. 
This is because, the ACC conjecture for an arbitrary finite set $I$ and 
the LSC conjecture imply the termination of flips \cite{Shokurov:letter}. 

The ACC conjecture is known for $d \le 2$ by Alexeev \cite{Alexeev} and Shokurov \cite{Shokurov:acc}, 
and for toric pairs by Ambro \cite{Ambro:toric}. 
Kawakita \cite{Kawakita:connectedness} proved the ACC conjecture on the interval $[1,3]$ for three-dimensional 
smooth varieties. 
Further, Kawakita \cite{Kawakita:discrete} proved that the ACC conjecture is true 
for fixed variety $X$ and a finite set $I$. 
More generally, he proved the discreteness of the set of log discrepancies for log triples 
(see Subsection \ref{subsection:mld} for the definition)
\[
\{ a_E (X, \Delta, \mfa) \mid \text{$(X, \Delta, \mfa)$ is lc, 
$\mfa \in I$, $E \in \mcD _X$} \} 
\]
when the pair $(X, \Delta)$ is fixed and $I$ is a finite set. 
Here, we denoted by $\mcD _X$ the set of all divisor over $X$. 
Further, $\mfa = \prod \mfa_i ^{r_i}$ is an $\mbR$-ideal sheaf with coefficients $r_i$ in $I$. 
The purpose of this paper is to generalize this results to the family of the varieties with fixed Gorenstein index. 

\begin{thm}\label{thm:main}
Fix $d \in \mbZ _{>0}$, $r \in \mbZ _{>0}$ and a finite subset $I \subset [0, + \infty)$. 
Then the following set 
\[
B(d,r,I) := 
\{ a_E (X, \mfa) \mid \text{$(X, \mfa) \in P(d,r)$, $\mfa \in I$, $E \in \mcD _X$}\} 
\subset [0, + \infty)
\]
is discrete in $\mbR$. 
Here we denote by $P(d,r)$ the set of all $d$-dimensional lc pairs 
$(X, \mfa)$ such that $r K_X$ is a Cartier divisor. 
\end{thm}

Since $\mld _x (X, \mfa) = a_E (X, \mfa)$ 
holds for some $E \in \mcD _X$, we get the following Corollary. 
\begin{cor}\label{cor:main}
Fix $d \in \mbZ _{>0}$, $r \in \mbZ _{>0}$ and a finite subset $I \subset [0, + \infty)$. 
Then the following set 
\begin{align*}
A'(d,r,I) := 
\{ \mld _x (X, \mfa) \mid \text{$X \in P(d,r)$, $\mfa \in I$, $x \in X$}\}  \subset [0, + \infty)
\end{align*}
is discrete in $\mbR$. 
Here we denote by $P(d,r)$ the set of all $d$-dimensional lc pairs 
$(X, \mfa)$ such that $r K_X$ is a Cartier divisor. 
\end{cor}

Corollary \ref{cor:main} does not imply the finiteness of $A'(d,r,I)$, because 
we do not know the boundedness of $A'(d,r,I)$. 
Hence Corollary \ref{cor:main} shows the finiteness of $A'(d,r,I)$ modulo 
the BDD (boundedness) conjecture, which states the boundedness of minimal log discrepancies. 
\begin{conj}[BDD conjecture]\label{conj:BDD}
For fixed $d \in \mbZ_{>0}$, there exists a real number $a(d)$ such that 
$\mld (X) \le a(d)$ holds for any $\mbQ$-Gorenstein $d$-dimensional normal variety $X$. 
\end{conj}
\noindent
The BDD conjecture is known only for $d \le 3$ \cite{Mark}. 
In arbitrary dimensions, the conjecture is known for the set of varieties with bounded multiplicity \cite{Kawakita:BDD}. 

As a corollary of Corollary \ref{cor:main}, we can prove the ACC for three-dimensional canonical pairs. 
\begin{cor}\label{cor:acc_3dimcan}
If $I \subset [0, 1]$ is a finite subset, the following set 
\[
\{ \mld _{x}(X, \Delta) \mid 
\text{$(X, \Delta)$ is a canonical pair, $\dim X = 3$, $\Delta \in I$, $x \in X$} \}, 
\]
denoted by $A_{\text{can}}(3,I)$, satisfies the ACC. 
Further, $1$ is the only accumulation point of $A_{\text{can}}(3,I)$. 
\end{cor}

Theorem \ref{thm:main} is proved by induction on $\dim _{\mbQ} \Span_{\mbQ} (I \cup \{ 1 \})$, 
the dimension of the $\mbQ$-vector space generated by $I \cup \{ 1 \}$. 
In the inductive step, we need the following theorem, 
which is about a perturbation of an irrational coefficient of log canonical pairs. 

\begin{thm}\label{thm:perturbe}
Fix $d \in \mbZ _{>0}$. Let $r_1, \ldots, r_{c'}$ be positive real numbers and let $r_0 = 1$. 
Assume that $r_0, \ldots, r_{c'}$ are $\mbQ$-linearly independent. 
Let $s_1, \ldots, s_c : \mbR^{c'+1} \to \mbR$ be $\mbQ$-linear functions from $\mbR ^{c'+1}$ to $\mbR$. 
Assume that $s_i (r_0, \ldots , r_{c'}) \in \mbR _{\ge 0}$ for each $i$. 
Then there exists a positive real number $\epsilon >0$ such that the following holds:
For any $\mbQ$-Gorenstein normal variety $X$ of dimension $d$ and $\mbQ$-Cartier effective 
Weil divisors $D_1, \ldots, D_c$ on $X$, 
if $(X, \sum_{1 \le i \le c} s_i(r_0, \ldots, r_{c'})D_i)$ is lc, then 
$(X, \sum_{1 \le i \le c} s_i(r_0, \ldots, r_{c' -1}, t)D_i)$ is also lc for any $t$ 
satisfying $|t - r_{c'}| \le \epsilon$. 
\end{thm}
\begin{rmk}
The positive real number $\epsilon$ in Theorem \ref{thm:perturbe} does not depend on $X$, 
but depends only on $d$, $r_1, \ldots, r_{c'}$, and $s_1, \ldots, s_c$. 
\end{rmk}

\noindent
Kawakita \cite{Kawakita:discrete} proved this theorem for a fixed variety $X$ using a method of generic limit, 
and prove the discreteness of log discrepancies for fixed $X$. 
When $c' = 1$ and each $s_i$ satisfies $s_i (\mbR _{\ge 0} ^2) \subset \mbR _{\ge 0}$, 
this theorem just states the rationality of accumulation points of log canonical thresholds proved by 
Hacon, M\textsuperscript{c}Kernan, and Xu \cite[Theorem 1.11]{HMX2}. 
Actually, the proof of Theorem \ref{thm:perturbe} heavily depends on their argument. 
We also note that the rationality of accumulation points of log canonical thresholds on smooth varieties 
was proved by Koll\'ar \cite[Theorem 7]{Kollar:which} and 
by de Fernex and Musta{\c{t}}{\u{a}} \cite[Corollary 1.4]{dFM:limit} using a method of generic limit.

The paper is organized as follows:
In Section \ref{section:pre}, we review some definitions and facts from 
the minimal model theory. 
Further we list some results on the ACC for log canonical thresholds by Hacon, M\textsuperscript{c}Kernan, and Xu \cite{HMX2}. 
In Section \ref{section:key}, we prove the key proposition (Theorem \ref{thm:accum_mfG}) 
which is necessary to prove Theorem \ref{thm:perturbe}. 
The essential idea of proof is due to the paper \cite{HMX2}. 
In Section \ref{section:perturbe}, we prove Theorem \ref{thm:perturbe}. 
In Section \ref{section:main}, we prove the main theorem (Theorem \ref{thm:main}) 
and the corollaries. 

\subsection*{Notation and convention}
Throughout this paper, we work over the field of complex numbers $\mbC$. 

\begin{itemize}
\item For an $\mbR$-divisor $D$ and a subset $I \subset \mbR$, 
we write $D \in I$ when all the non-zero coefficients of $D$ belong to $I$. 
\item For an $\mbR$-ideal sheaf $\mfA = \prod \mfa _i ^{r_i}$ and a subset $I \subset \mbR$, 
we write $\mfA \in I$ when all the non-zero coefficients $r_i$ of $\mfA$ belong to $I$. 
\end{itemize}

\section{Preliminaries}\label{section:pre}
\subsection{Minimal log discrepancies}\label{subsection:mld}
We recall some notations in the theory of singularities in the minimal model program. 
For more details we refer the reader \cite{KM}. 

A \textit{log pair} $(X, \Delta)$ is a normal variety $X$ and an effective $\mbR$-divisor $\Delta$ such that $K_X + \Delta$ is 
$\mbR$-Cartier. 
If $X$ is $\mbQ$-Gorenstein, we sometimes identify $X$ with the log pair $(X, 0)$. 

An \textit{$\mbR$-ideal sheaf} on $X$ is a formal product $\mfa _1 ^{r_1} \cdots \mfa _s ^{r_s}$, 
where $\mfa _1 , \ldots , \mfa _s $ are ideal sheaves on $X$ and $r_1 , \ldots , r_s$ 
are positive real numbers. 
For a log pair $(X, \Delta)$ and an $\mbR$-ideal sheaf $\mfa$, 
we call $(X, \Delta, \mfa)$ a \textit{log triple}. 
When $\Delta = 0$ (resp.\  $\mfA = \mcO _X$), 
we sometimes drop $\Delta$ (resp.\ $\mfA$) and write $(X, \mfa)$ (resp.\ $(X, \Delta)$). 

For a proper birational morphism $f: X' \to X$ from a normal variety $X'$ 
and a prime divisor $E$ on $X'$, the \textit{log discrepancy} of $(X, \Delta, \mfa)$ 
at $E$ is defined as
\[
a_E (X, \Delta, \mfa) := 1 + \coeff _E (K_{X'} - f^* (K_X + \Delta)) -\ord _E \mfa, 
\]
where $\ord _E \mfa := \sum _{i=1} ^s r_i \ord _E \mfa _i$. 
The image $f(E)$ is called the \textit{center of $E$ on $X$}, and we denote it by $\cent _X (E)$. 
For a closed subset $Z$ of $X$, the \textit{minimal log discrepancy} (\textit{mld} for short) over $Z$ is 
defined as 
\[
\mld _Z (X, \Delta, \mfa) := \inf _{\cent _X(E) \subset Z} a_E (X, \Delta, \mfa). 
\]
In the above definition, 
the infimum is taken over all prime divisors $E$ on $X'$ with the center $\cent _X (E) \subset Z$, 
where $X'$ is a higher birational model of $X$, that is, 
$X'$ is the source of some proper birational morphism $X' \to X$.

\begin{rmk}\label{rmk:attain}
It is known that $\mld _Z (X, \Delta, \mfa)$ is in $\mbR _{\ge 0} \cup \{ - \infty \}$ and that
if $\mld _Z (X, \Delta, \mfa) \ge 0$, then the infimum on the right hand side in the definition is actually 
the minimum. 
\end{rmk}
\begin{rmk}\label{rmk:cartier}
Let $D_i$ be effective Weil divisors on $X$, 
and $\mfa_i := \mcO _X (- D_i)$ the corresponding ideal sheaves. 
When $X$ is $\mbQ$-Gorenstein and $D_i$ are Cartier divisors, 
we can identify $(X, \sum r_i D_i)$ and $(X, \prod \mfa_i ^{r_i})$. 
Indeed, for any divisor $E$ over $X$, we have $a_E (X, \sum r_i D_i) = a_E (X, \prod \mfa_i ^{r_i})$. 
\end{rmk}

For simplicity of notation, 
we write $\mld _x (X, \Delta, \mfa)$ instead of $\mld _{\{x\}} (X, \Delta, \mfa)$ 
for a closed point $x$ of $X$, and write 
$\mld (X, \Delta, \mfa)$ instead of $\mld _X (X, \Delta, \mfa)$. 

We say that the pair $(X, \Delta, \mfa)$ is \textit{log canonical} (\textit{lc} for short) 
if $\mld (X, \Delta, \mfa) \ge 0$. 
Further, we say that the pair $(X, \Delta, \mfa)$ is \textit{Kawamata log terminal} 
(\textit{klt} for short) 
if $\mld (X, \Delta, \mfa) > 0$. 
When $E$ is a divisor over $X$ such that $a_E(X, \Delta, \mfa) \le 0$, 
the center $\cent _X(E)$ is called a \textit{non-klt center}. 

We say that the pair $(X, \Delta, \mfa)$ is \textit{canonical} (resp.\ \textit{terminal}) 
if $a_E (X, \Delta, \mfa) \ge 1$ (resp.\ $> 1$) for any exceptional divisor $E$ over $X$.

\subsection{Extraction of divisors}
In this subsection, we recall some known results on extractions of divisors.

We can extract a divisor whose log discrepancy is at most one. 
\begin{thm}\label{thm:extraction}
Let $(X, \Delta)$ be a klt pair, and 
let $E$ be a divisor over $X$ such that $a_E(X, \Delta) \le 1$. 
Then there exists a projective birational morphism $\pi : Y \to X$ such that 
$Y$ is $\mbQ$-factorial and the only exceptional divisor is $E$. 
\end{thm}
\begin{proof}
This is the special case of \cite[Corollary 1.4.3]{BCHM}. 
\end{proof}

When $(X, \Delta)$ is lc, we can find a modification which is dlt. 
We call a log pair $(X, \Delta)$ 
\textit{divisorial log terminal} (\textit{dlt} for short) 
when there exists a log resolution $f : Y \to X$ such that 
$a_E(X, \Delta) > 0$ for any $f$-exceptional divisor $E$ on $Y$. 

\begin{thm}[dlt modification]\label{thm:dlt_modification}
Let $(X, \Delta)$ be a lc pair. Then there exists a projective birational morphism $f : Y \to X$ 
with the following properties:
\begin{itemize}
\item $Y$ is $\mbQ$-factorial. 
\item $(Y, \Delta _Y)$ is dlt, where we define $\Delta _Y$ as $K_Y + \Delta _Y = f^* (K_X + \Delta)$. 
\item $a_E (X, \Delta) = 0$ for every $f$-exceptional divisor $E$. 
\end{itemize}
\end{thm}
\begin{proof}
See \cite[Theorem 10.4]{Fujino:fundamental} for instance. 
\end{proof}

\subsection{ACC for log canonical thresholds}
In Section \ref{section:key}, 
we need the following ACC properties proved by Hacon, M\textsuperscript{c}Kernan, and Xu \cite{HMX2}. 

\begin{thm}[Hacon, M\textsuperscript{c}Kernan, Xu {\cite[Theorem 1.4]{HMX2}}]\label{thm:lct}
Fix $d \in \mbZ _{>0}$ and a subset $I \subset [0,1]$ satisfying the DCC. 

Then there is a finite subset $I_0 \subset I$ with the following property:
If $(X, \Delta)$ is a log pair such that 
\begin{itemize}
\item $(X, \Delta)$ is lc, $\dim X = d$, $\Delta \in I$, and
\item there exists a non-klt center $Z \subset X$ which is contained in every component of $\Delta$, 
\end{itemize}
then $\Delta \in I_0$. 

\end{thm}
\begin{thm}[Hacon, M\textsuperscript{c}Kernan, Xu {\cite[Theorem 1.5]{HMX2}}]\label{thm:num_triv}
Fix $d \in \mbZ _{>0}$ and a subset $I \subset [0,1]$ satisfying the DCC. 

Then there is a finite subset $I_0 \subset I$ with the following property:
If $(X, \Delta)$ is a projective log pair such that 
\begin{itemize}
\item $(X, \Delta)$ is lc, $\dim X = d$, $\Delta \in I$, and
\item $K_X + \Delta \equiv 0$, 
\end{itemize}
then $\Delta \in I_0$. 
\end{thm}

\section{Accumulation points of log canonical thresholds}\label{section:key}
The goal of this section is to prove Corollary \ref{cor:rationality}. 
It is a generalization of \cite[Theorem 1.11]{HMX2} and 
necessary for the proof of Theorem \ref{thm:perturbe}. 

Usually, the log canonical threshold is defined as follows: 
for a lc pair $(X, \Delta)$ and a $\mbQ$-Cartier $\mbZ$-Weil effective divisor $M$, 
\[
\LCT (\Delta; M) := \sup \{ c \in \mbR _{\ge 0} \mid \text{$(X, \Delta + cM)$ is lc} \}. 
\]
However, for the proof of Theorem \ref{thm:perturbe}, 
we need to treat the case when $M$ is not effective. 
According to this reason, we introduce the new threshold set $\mfL _d(I)$. 
It no longer satisfies the ACC, 
but we can prove the rationality of the accumulation points (Corollary \ref{cor:rationality}). 

Corollary \ref{cor:rationality} easily follows from 
Theorem \ref{thm:local_global} and Theorem \ref{thm:accum_mfG}. 
They are proved in essentially 
the same way of the proof of Proposition 11.5 and Proposition 11.7 in \cite{HMX2}. 
For the reader's convenience, we follow the proof of Proposition 11.5 and Proposition 11.7 in \cite{HMX2}, 
and use as same notations as possible.

First, we introduce some notations. 
For a subset $I \subset [0, + \infty)$, we define $I_+$ as follows:
\[
I_+ := \{0\} \cup \big\{ \sum _{1 \le i \le l} r_i \mid l \in \mbZ _{> 0}, r_1, \ldots , r_l \in I \big\}. 
\]
This becomes a discrete set if $I$ is discrete. 
When $D_i$ are finitely many distinct prime divisors and 
$d_i(t) : \mbR \to \mbR$ are $\mbR$-linear functions, 
then we call the formal finite sum $\sum _i d_i (t) D_i$ a \textit{linear functional divisor}. 
\begin{defi}[$\mcD _c(I)$]\label{defi:D_c}
Fix $c \in \mbR _{\ge 0}$ and a subset $I \subset [0, + \infty)$. 
For a linear functional divisor $\Delta (t) = \sum _i d_i (t) D_i$, 
we write $\Delta (t) \in \mcD _c(I)$ when the following conditions are satisfied:
\begin{itemize}
\item Each $d_i(t)$ is equal to $1$ or the form of $\frac{m-1+f+kt}{m}$, 
where $m \in \mbZ _{>0}$, $f \in I_+$, and $k \in \mbZ$. 
\item Further, $f+kt$ above can be written as $f+kt = \sum_j (f_j + k_j t)$, 
where $f_j \in I \cup \{ 0 \}$, $k_j \in \mbZ$, and $f_j + k_j c \ge 0$ hold for each $j$. 
\end{itemize}
Further, by abuse of notation, we also write $d_i(t) \in \mcD _c(I)$ if $d_i (t)$ 
satisfies the above conditions. 
\end{defi}

The form of the coefficient $d_i(t)$ is preserved by adjunction. 
\begin{lem}\label{lem:adjunction}
Fix $c \in \mbR _{\ge 0}$ and a subset $I \subset [0,1]$. 
Let $X$ be a $\mbQ$-factorial normal variety and 
$\Delta (t) = \sum _{0 \le i \le c} d_i (t) D_i$ be a linear functional divisor on $X$. 
Assume the following conditions: 
\begin{itemize}
\item $\Delta(t) \in \mcD _c(I)$, and $(X, \Delta (c))$ is lc.  
\item $d_0(t) = 1$, and $d_i (c) > 0$ for each $i$.  
\end{itemize}
Let $S^{\text{n}}$ be the normalization of $S := D_0$. 
Define a linear functional divisor $\Delta _{S^{\text{n}}} (t)$ on $S^{\text{n}}$ by adjunction:
\[
(K_X + \Delta (t)) |_{S^{\text{n}}} = K_{S^{\text{n}}} + \Delta _{S^{\text{n}}} (t). 
\]

Then, $\Delta _{S^{\text{n}}}(t) \in \mcD _c(I)$ holds. 
\end{lem}
\begin{proof}
The statement follows from \cite[Proposition 16.6]{Kollars}. 
We give a sketch of proof. 

Let $p \in S$ be a codimension one point of $S$. 

Suppose that $(X, D_0)$ is not plt at $p$. 
Then $p \not \in \Supp D_i$ for any $i \ge 1$ and 
$\coeff _p \Diff _{S^{\text{n}}} (0) = 0 \ \text{or}\ 1$ \cite[Proposition 16.6.1-2]{Kollars}. 
Hence, we have $\coeff _p \Delta _{S^{\text{n}}} (t) = 0 \ \text{or}\ 1$ for any $t$. 

Suppose that $(X, D_0)$ is plt at $p$. 
Then $\coeff _p \Diff _{S^{\text{n}}} (0) = \frac{m-1}{m}$ holds for some $m \in \mbZ _{>0}$, and 
$mD$ becomes Cartier at $p$ for any Weil divisor $D$ \cite[Proposition 16.6.3]{Kollars}. 
Hence, $\coeff _p \Delta _{S^{\text{n}}} (t)$ is the form of 
\[
\frac{m-1}{m} + \frac{1}{m} \sum _j \frac{n_j - 1 + f_j + k_j t}{n_j}, 
\]
where $\frac{n_j - 1 + f_j + k_j t}{n_j}$ is the form as in the definition of $\mcD _c (I)$. 
We can prove that such form also satisfies the condition in the definition of $\mcD _c (I)$ 
by easy calculation (cf.\ \cite[Lemma 4.4]{MP}).
\end{proof}

We define $\mfL_{d}(I)$, the set of all log canonical thresholds derived from coefficients $I$. 

\begin{defi}[$\mfL_{d}(I)$]\label{defi:L}
Let $d \in \mbZ _{>0}$ and let $I \subset [0, + \infty)$ be a subset. 
We define $\mfL_d (I) \subset \mbR_{\ge 0}$ as follows: 
$c \in \mfL_d (I)$ if and only if
there exist a $\mbQ$-Gorenstein normal varieties $X$, and a linear functional divisor 
$\Delta (t)$ with the following conditions:
\begin{itemize}
\item $\dim X \le d$, $\Delta (t) \in \mcD _c(I)$, 
\item $\Delta (a)$ is $\mbR$-Cartier for any $a \in \mbR$, 
\item $(X, \Delta(c))$ is lc, and 
\item $(X, \Delta(c + \epsilon))$ is not lc for any $\epsilon > 0$, or 
$(X, \Delta (c - \epsilon))$ is not lc for any $\epsilon > 0$. 
\end{itemize}
\end{defi}
\begin{rmk}
When we say that $(X, \Delta)$ is a lc pair, we assume that $\Delta$ is effective. 
Therefore, we say that $(X, \Delta)$ is not lc when $\Delta$ is not effective. 
\end{rmk}


Further, we define $\mfG_{d}(I)$, the set of all numerically trivial thresholds derived from coefficients $I$. 
\begin{defi}[$\mfG_d(I)$]\label{defi:G}
Let $d \in \mbZ _{>0}$ and let $I \subset [0, + \infty)$ be a subset. 
We define $\mfG_d (I) \subset \mbR_{\ge 0}$ as follows: $c \in \mfG_d (I)$ if and only if
there exist a $\mbQ$-factorial normal projective variety $X$, and a linear functional divisor 
$\Delta (t)$ with the following conditions:
\begin{itemize}
\item $\dim X \le d$, $\Delta (t) \in \mcD _c(I)$, 
\item $(X, \Delta(c))$ is lc, and $K_X + \Delta (c) \equiv 0$. 
\item $K_X + \Delta (c') \not \equiv 0$ for some $c' \not = c$ (equivalently for all $c' \not = c$). 
\end{itemize}
\end{defi}

By the following theorem, we can reduce a local problem to a global problem. 
\begin{thm}\label{thm:local_global}
Let $d \ge 2$ and $I \subset [0, + \infty)$ be a subset. 
Then, $\mfL_d (I) \subset \mfG_{d-1}(I)$ holds. 
\end{thm}
\begin{lem}\label{lem:example}
Let $c \in \mbR _{\ge 0}$ and $I \subset [0, + \infty)$ be a subset. 
Suppose that there exists an $\mbR$-linear function $d (t) : \mbR \to \mbR$ 
with the following conditions:
\begin{itemize}
\item $d (t) \in \mcD _c (I)$, and $d (t)$ is not a constant function. 
\item $d (c) = 0 \ \text{or}\ 1$. 
\end{itemize}
Then $c \in \mfG _d (I)$ for any $d \ge 1$. 
Especially, $\frac{f}{k} \in \mfG _d (I)$ holds for any $d \ge 1$, $f \in I \cup \{ 0 \}$, and $k \in \mbZ _{>0}$. 
\end{lem}
\begin{proof}
We can easily construct on a curve.
\end{proof}
\begin{proof}[Proof of Theorem \ref{thm:local_global}]
Let $c \in \mfL_d(I)$, and let $(X, \Delta(t))$ be as in Definition \ref{defi:L}. 
Assume that $(X, \Delta (c + \epsilon))$ is not lc for any $\epsilon > 0$ 
(the same proof works in the other case). 
We may write $\Delta (t) = \sum _{i} d_i (t) D_i$ with distinct prime divisors $D_i$. 
By Lemma \ref{lem:example}, we may assume that $d_i (c) > 0$ for any $i$. 
Then $\Delta (c + \epsilon) \ge 0$ holds for sufficiently small $\epsilon >0$. 

Let $f: Y \to X$ be a dlt modification (Theorem \ref{thm:dlt_modification}) of $(X, \Delta (c))$. 
Then $Y$ is $\mbQ$-factorial and we can write
\[
K_Y + T + \Delta' (c) = f^* (K_X + \Delta (c)), 
\]
where $\Delta ' (t)$ is the strict transform of $\Delta (t)$, and 
$T$ is the sum of the exceptional divisors. 
Since the pair $(Y, T + \Delta ' (c))$ is dlt, there exists a divisor 
$E$ on $Y$ such that 
\[
a_E (X, \Delta (c)) = 0, \quad a_E (X, \Delta (c + \epsilon)) < 0
\]
for any $\epsilon > 0$. 
If $E$ is not $f$-exceptional, 
then $d_i (c) = 1$ holds for some $d_i(t)$ which is not identically one. 
In this case $c \in \mfG_{d-1}(I)$ by Lemma \ref{lem:example}. 

In what follows, we assume that $E$ is $f$-exceptional and so a component of $\Supp T$. 
By adjunction, we can define a linear functional divisor 
$\Delta_E (t)$ on $E$ such that
\[
(K_Y + T + \Delta' (t))|_E = K_E + \Delta _E (t). 
\]
Here, $\Delta _E (t) \in \mcD _c(I)$ holds by Lemma \ref{lem:adjunction}.

Let $F$ be a general fiber of $E \to f(E)$. 
Define $\Delta _F(t)$ as
\[
(K_E + \Delta _E(t))|_F = K_F + \Delta _F(t). 
\]
Then $(F, \Delta _F (t))$ satisfies
\begin{itemize}
\item $\dim F \le d-1$, $F$ is projective, 
\item $\Delta _F (t) \in \mcD _c (I)$,  
\item $K_F + \Delta _F (c) = f^*(K_X + \Delta (c)) |_F \equiv 0$, and
\item $(F, \Delta _F (c))$ is lc. 
\end{itemize}
Hence $(F, \Delta _ F(t))$ satisfies all conditions in Definition \ref{defi:G} except for 
$K_F + \Delta _F (c') \not \equiv 0$ for some $c'$. 

We may write $\Delta(t) = \Delta + t M$ with an $\mbR$-divisor $\Delta$ and a $\mbQ$-divisor $M$. 
Write $M = M_+ - M_-$, where $M_+ \ge 0$ and $M_- \ge 0$ have no common components. 
Since $a_E (X, \Delta + (c + \epsilon)M) < a_E (X, \Delta + cM) = 0$, 
it follows that $\ord _E M_+ > \ord _E M_- \ge 0$. 
Possibly replacing $E$ by other component of $T$, 
we may assume that 
\[
\ord _E M_-  \cdot \ord _{E_j} M_+  \le \ord _{E_j}M_- \cdot \ord _E M_+
\]
for any component $E_j \subset \Supp T$. 
We may take $\epsilon _1 \ge \epsilon _2 >0$ such that 
$a_E(X, \Delta + (c+\epsilon _1)M - \epsilon _2 M_+) = 0$. 
Note that 
\[
\epsilon _1 (\ord _E M_+ - \ord _E M_-) = \epsilon _2 \ord _E M_+
\]
holds. 
Then we have 
\begin{align*}
0 &\equiv f^* (K_X + \Delta + (c+\epsilon _1)M - \epsilon _2 M_+)|_F \\
&= (K_Y + T + U + \Delta ' + (c + \epsilon _1)M' - \epsilon _2 M_+ ' )|_F \\
&= K_F + \Delta _F (c + \epsilon_1) + U|_F - \epsilon _2 M_+ ' |_F, 
\end{align*}
where we set 
\[
U 
=  \sum_j (\epsilon _1 \ord _{E_j} M - \epsilon _2 \ord _{E_j} M_+ )E_j. 
\]
Note that 
\begin{align*}
&\epsilon _1 \ord _{E_j} M - \epsilon _2 \ord _{E_j} M_+ \\
=& \ 
\frac{\epsilon _1}{\ord _{E} M_+} 
(\ord _E M_-  \cdot \ord _{E_j} M_+ - \ord _{E_j}M_- \cdot \ord _E M_+)\\ 
\le& \  0. 
\end{align*}
Therefore 
$K_F + \Delta _F (c + \epsilon_1) \equiv \epsilon _2 M_+ ' |_F - U|_F \ge \epsilon _2 M_+ ' |_F$. 
Since $\ord _E M_+ > 0$, it follows that $f(E) \subset \Supp M_+$ and so $M_+ ' |_F > 0$. 
Therefore $K_F + \Delta _F (c + \epsilon_1)$ is not numerically trivial. 
\end{proof}

\begin{thm}\label{thm:accum_mfG}
Let $d \ge 2$ and let $I \subset [0, + \infty)$ be a finite subset. 
The accumulation points of $\mfG_d (I)$ are contained in $\mfG_{d-1} (I)$. 
\end{thm}

As a corollary, we can prove the rationality of 
the accumulation points of $\mfL_d (I)$.

\begin{cor}\label{cor:rationality}
Let $d \in \mbZ _{>0}$ and let $I \subset [0, + \infty)$ be a finite subset. 
The accumulation points of $\mfL _d (I)$ are contained in $\Span _{\mbQ} (I \cup \{ 1 \})$, 
where we denote by $\Span _{\mbQ} (I \cup \{ 1 \}) \subset \mbR$ the 
$\mbQ$-vector space spanned by the elements of $I$ and $1$. 
\end{cor}

We prove a stronger statement (cf.\ \cite[Proposition 11.7]{HMX2}). 
\begin{prop}\label{prop:accum_A}
Let $d \ge 2$ and let $I \subset [0, + \infty)$ be a finite subset. 
Further, let $c \in \mbR _{\ge 0}$. 

Suppose that for each $i \in \mbZ _{>0}$, 
there exist $c_i \in \mbR_{\ge 0}$, 
a $\mbQ$-factorial normal projective variety $X_i$, and 
a linear functional divisor $\Delta _i (t)$ on $X_i$ 
with the following conditions:
\begin{itemize}
\item The sequence $c_i$ is increasing or decreasing. 
Further, $c_i$ is accumulating to $c$. 
\item $\dim X_i \le d$ for each $i$. 
\item $\Delta _i (t)$ can be written as $\Delta _i (t) = A_i + B_i (t)$, 
where the coefficients of $A_i$ are approaching one, and $B_i (t) \in \mcD _{c_i} (I)$. 
\item $(X_i, \Delta_i (c_i))$ is lc, and $K_{X_i} + \Delta _i (c_i) \equiv 0$. 
\item $K_{X_i} + \Delta _i (c' _i) \not \equiv 0$ for some $c_i ' \not = c_i$. 
\end{itemize}

Then, $c \in \mfG_{d-1} (I)$ holds. 
\end{prop}

\begin{rmk}
If $c_i \in \mfG_{d}(I)$, then $c_i$ satisfies the above conditions (In this case, $A_i =0$). 
Hence, Theorem \ref{thm:accum_mfG} follows from Proposition \ref{prop:accum_A}. 
\end{rmk}

In the proof of Proposition \ref{prop:accum_A}, we reduce to the case when $X_i$ has Picard number one, 
and apply the following lemma from \cite{HMX2}. 

\begin{lem}[{\cite[Lemma 11.6]{HMX2}}]\label{lem:picard1}
Let $(X, \Delta)$ be a projective $\mbQ$-factorial lc pair of dimension $d$ and of Picard number one. 
Assume that $K_X + \Delta \equiv 0$. 
If the coefficients of $\Delta$ are at least $\delta > 0$, then $\Delta$ has at most 
$\frac{d+1}{\delta}$ components. 
\end{lem}

%

\begin{proof}[Proof of Proposition \ref{prop:accum_A}]
Possibly replacing $A_i$ and $B_i (t)$, 
we may assume that the coefficient of $B_i (t)$ is not identically one. 
We may write $B_i (t) = \sum _l d_{il} (t) D_{il}$ 
as in Definition \ref{defi:D_c}. 

By Lemma \ref{lem:example}, we may assume that 
$(I \cup \{ 0 \}) \cap c \mbZ_{>0} =  \emptyset$. 
Then we may assume the following conditions on $B_i(t)$. 
\begin{lem}\label{lem:coeff}
We may assume the following conditions:
\begin{itemize}
\item[(1)] When we write $d_{il}(t) = \frac{m-1+f+kt}{m}$ as in Definition \ref{defi:D_c}, $f$ and $k$ have only finitely many possibilities. 
\item[(2)] $d_{il} (c_i)$ are bounded from zero, and $d_{il} (c_i) < 1$ for any $i, l$. 
\item[(3)] $d_{il} (c) > 0$. 
\item[(4)] The set $\{ d_{il} (c)  \mid i,l \}$ satisfies the DCC. 
\end{itemize}
\end{lem}
\begin{proof}
Since $(I \cup \{ 0 \}) \cap c \mbZ_{>0} = \emptyset$, possibly passing to a tail of the sequence, 
we may assume that there exist $k' \in \mbZ _{>0}$ and $\epsilon \in \mbR _{>0}$ such that 
for any $f_j \in I \cup \{ 0 \}$, $k_j \in \mbZ$, and $i$, 
\begin{itemize}
\item $f_j + k_j c_i \ge 0$ implies $f_j + k_j c_i \ge \epsilon$ and $k_j \ge - k'$ unless $f_j = k_j =0$. 
\end{itemize}
Here, we note that $I$ is a finite set. 

Let $d_{il}(t) = \frac{m-1+f+kt}{m}$ be a coefficient of $B_i(t)$. 
By assumption, $f+kt$ above can be written as $f+kt = \sum _j (f_j + k_jt)$, 
where $f_j \in I$, $k_j \in \mbZ$, and $f_j + k_j c_i \ge 0$ hold for each $j$.

Note that $f+kc_i \le 1$ by the log canonicity. 
Since $f_j + k_j c_i \ge 0$ implies $f_j + k_j c_i \ge \epsilon$ and $k_j \ge - k'$, 
it follows that $k$ is bounded from below. 
Since $c_i \ge \epsilon$, it follows that $k$ is also bounded from above. 
As the set $I_+$ is discrete, $f$ has also only finitely many possibilities. 
Therefore (1) follows. 

By (1), it follows that $d_{il}(c_i) \ge \min \{ \frac{1}{2}, \epsilon \}$. 
Hence $d_{il} (c_i)$ are bounded from zero. 
Since $c_i$ are distinct, by (1), possibly passing to a subsequence, we may assume that 
$d_{il} (c_i) \not= 1$ (hence $d_{il} (c_i) < 1$) holds for any $i, l$. 
Thus, (2) follows. 

(3) follows from (2) and (4) follows from (1). 
\end{proof}
By Lemma \ref{lem:coeff} (2), possibly passing to a tail of the sequence, we may assume that $A_i$ and $B_i (t)$ have no common components, 
and that $\lfloor A_i \rfloor = \lfloor A_i + B_i(c_i) \rfloor$. 

In our setting, the following claim is important and 
allow the same argument in \cite{HMX2} to work. 

\begin{claim}\label{claim:lc}
We may assume that $(X_i, \lceil A_i \rceil + B _i (c))$ is lc for any $i$. 
\end{claim}
\begin{proof}
We may write $\Delta _i (t) = A_i + M_i + t(N^+ _i - N^- _i)$, 
where $N^+_i \ge 0$ and $N^-_i \ge 0$ have no common components. 
$(X_i, A_i + M_i + c_i (N^+_i - N^-_i))$ is lc by the assumption. 

First suppose $c_i < c$. 
Note that $A_i + M_i + c_i N^+_i - cN^-_i \ge 0$ (Lemma \ref{lem:coeff} (3)). 
Hence $(X_i , A_i + M_i + c_i N^+_i - c N^-_i)$ is also lc. 
Here, the coefficients of $M_i -  c N^-_i$ satisfy the DCC (Lemma \ref{lem:coeff} (4)), 
and the coefficient of $A_i$ and the sequence $c_i$ are increasing. 
Hence by Theorem \ref{thm:lct}, possibly passing to a tail of the sequence, 
we may assume that $(X_i , \lceil A_i \rceil + M_i + c N^+_i - c N^-_i)$ is lc. 

Suppose $c_i > c$. Then $(X_i, A_i + M_i + c N^+_i - c_i N^-_i)$ is lc. 
Here, the coefficients of $M_i + c N^+_i$ satisfy the DCC (Lemma \ref{lem:coeff} (4)), 
and the coefficient of $A_i$ and the sequence $-c_i$ are increasing. 
Hence by Theorem \ref{thm:lct}, possibly passing to a tail of the sequence, 
we may assume that $(X_i , \lceil A_i \rceil + M_i + c N^+_i - c N^-_i)$ is lc. 
\end{proof}

Set $a_i := \mld (X_i, \Delta_i (c_i)) \ge 0$. 
Possibly passing to a subsequence, it is sufficient to treat the following two cases:
\begin{enumerate}
\item[(A)] $a_i$ is bounded away from zero. 
\item[(B)] $a_i$ approaches zero. 
\end{enumerate}

\vspace{2mm}

\noindent
\underline{\bf{Case B}}\ \ We treat the case when $a_i$ approaches zero from above. 

\vspace{2mm}

\noindent
\underline{\bf{STEP B-1}}\ \ 
We reduce to the case when $A_i \not = 0$ and $(X_i, \Delta _i (c_i))$ is dlt. 

We may assume $a_i \le 1$ for any $i$. 
Take an extraction $\pi_i:X_i ' \to X_i$ of a divisor $E_i$ computing 
$\mld (X_i, \Delta _i(c_i)) = a_i$ (Theorem \ref{thm:extraction} and \ref{thm:dlt_modification}). 
Then we may write 
\[
K_{X_i '} + (1 - a_i)E_i + T_i + \Delta _i ' (c_i) = \pi _i ^* (K_{X_i} + \Delta _i (c_i)), 
\]
where $T_i$ is the sum of exceptional divisors (Note that $T_i = 0$ when $a_i > 0$) and 
$\Delta _i ' (t)$ is the strict transform of $\Delta _i (t)$. 
Then $\big( X_i ', (1 - a_i)E_i + T_i + \Delta _i ' (t) \big)$ satisfies the following conditions: 
\begin{itemize}
\item We may write $(1 - a_i)E_i + T_i + \Delta _i ' (t) = A' _i + B' _i (t)$ with all the conditions in 
Proposition \ref{prop:accum_A}. 
\item $\lfloor A_i' \rfloor = \lfloor A_i' + B_i'(c_i) \rfloor$ and $A' _i \not = 0$. 
\item $\big( X_i ', (1 - a_i)E_i + T_i + \Delta _i ' (c_i) \big)$ is dlt. 
\end{itemize}
Hence, we may replace $(X_i, \Delta _i (t))$ by 
$\big( X_i ', (1 - a_i)E_i + T_i + \Delta _i ' (t) \big)$. 

\vspace{2mm}

\noindent
\underline{\bf{STEP B-2}}\ \ We are done if there exists a component 
$S_i \subset \Supp \lfloor A_i \rfloor$
such that $(K_{X_i} + \Delta _i (c_i ')) |_{S_i} \not \equiv 0$. 

Suppose that there exists a component $S_i \subset \Supp \lfloor A_i \rfloor$
such that $(K_{X_i} + \Delta _i (c_i ')) |_{S_i} \not \equiv 0$. 
By adjunction, we can define $\Delta _{S_i} (t)$ as follows:
\[
(K_{X_i} + \Delta _i (t))|_{S_i} = K_{S_i} + \Delta _{S_i} (t). 
\]
Then $(S_i, \Delta _{S_i} (t))$ satisfies the following conditions: 
\begin{itemize}
\item $\dim S_i \le d-1$. 
\item $K_{S_i} + \Delta _{S_i} (c_i ') \not \equiv 0$ by the assumption. 
\item $(S_i, \Delta _{S_i} (t))$ satisfies the other conditions in Proposition \ref{prop:accum_A}. 
\end{itemize}
Hence, we may replace $(X_i, \Delta _i (t))$ by $(S_i, \Delta _{S_i} (t))$. 
By induction on $d$, it follows that $c \in \mfG_{d-2} (I) \subset \mfG_{d-1} (I)$. 

\vspace{2mm}

\noindent
\underline{\bf{STEP B-3}}\ \ 
We are done if 
$f_i : X_i \to Z_i$ is a Mori fiber space with $\dim Z_i > 0$, and 
$\Supp A_i$ dominates $Z_i$. 

Let $F_i$ be the general fiber of $f_i$. We may define $\Delta _{F_i} (t)$ as follows:
\[
(K_{X_i} + \Delta _i (t))|_{F_i} = K_{F_i} + \Delta _{F_i} (t). 
\]

Then $(F_i, \Delta _{F_i} (t))$ satisfies all conditions in Proposition \ref{prop:accum_A} 
except for $K_{F_i} + \Delta _{F_i} (c_i') \not \equiv 0$. 
Hence, if $K_{F_i} + \Delta _{F_i} (c_i') \not \equiv 0$ for some $c_i '$, then 
$c \in \mfG_{d-2} (I) \subset \mfG_{d-1} (I)$ by induction on $d$. 

Suppose that $K_{F_i} + \Delta _{F_i} (c_i') \equiv 0$, and so $K_{F_i} + \Delta _{F_i} (c) \equiv 0$. 
We may write $\Delta _{F_i} (t) = A' _i + B' _i (t)$ with the conditions in Proposition \ref{prop:accum_A}.  
Note that $(F_i, \Delta _{F_i} (c))$ is lc by the same reason as Claim \ref{claim:lc}. 
Since the coefficients of $B_i ' (c)$ satisfies the DCC (Lemma \ref{lem:coeff} (4)), 
it follows that $\lfloor A' _i \rfloor = A' _i$ by Theorem \ref{thm:num_triv}. 
Therefore, there exists a component $S_i \subset \Supp \lfloor A_i \rfloor$ such that $f(S_i) = Z_i$. 
Since $(K_{X_i} + \Delta _i (c))|_{F_i} \equiv 0$, 
$K_{X_i} + \Delta _i (c)$ is linearly equivalent 
to the pulled back of an $\mbR$-divisor $D_i$ on $Z_i$. 
As $K_{X_i} + \Delta _i (c) \not \equiv 0$, it follows that $D_i \not \equiv 0$, 
and so $(K_{X_i} + \Delta _i (c)) |_{S_i} \not \equiv 0$. 
Therefore we are done by STEP B-2. 

\vspace{2mm}

\noindent
\underline{\bf{STEP B-4}}\ \ 
We finish the case when $(X_i, \Delta _i (c_i))$ is not klt 
(equivalently $\lfloor A_i \rfloor \not= 0$ by STEP B-1). 

Suppose that $(X_i, \Delta _i (c_i))$ is not klt. 
Then $\lfloor A_i \rfloor \not= 0$. 
We run a $(K_{X_i} + \Delta _i (c_i) - \lfloor A_i \rfloor)$-MMP. 
Since $K_{X_i} + \Delta _i (c_i) - \lfloor A_i \rfloor \equiv - \lfloor A_i \rfloor$ is not pseudo-effective, 
a $(K_{X_i} + \Delta _i (c_i) - \lfloor A_i \rfloor)$-MMP terminates and 
ends with a Mori fiber space by \cite[Corollary 1.3.3]{BCHM}. 

Let $f_i : X_i \dashrightarrow X_i '$ be a step of the MMP. 
First suppose that $f_i$ is birational. 
We write 
\[
A_i' := f_{i*} A_i, \quad B_i ' (t) := f_{i*} B_i (t), \quad  \Delta' _i (t) = A_i' + B_i ' (t). 
\]
Then, $(X_i ', \Delta' _i (t))$ satisfies all conditions in Proposition \ref{prop:accum_A} except for 
$K_{X_i'} + \Delta' _i (c_i ') \not \equiv 0$. 

Assume that $K_{X_i'} + \Delta' _i (c_i ') \equiv 0$ (Hence, $f_i$ is a divisorial contraction). 
Set $D_i := K_{X_i} + \Delta _i (c_i') \not \equiv 0$. 
We may write $D_i - f_i ^* f_{i*} D_i = a E$, where $E$ is the $f_i$-exceptional divisor, and $a \in \mbR$. 
Since $D_i \not \equiv 0$ and $f_{i*} D_i \equiv 0$, we have $a E \not \equiv 0$. 
As $f_i$ is $\lfloor A_i \rfloor$-positive, there exists a component 
$T_i \subset \Supp \lfloor A_i \rfloor$ such that 
$E |_{T_i} \not \equiv 0$. Therefore, we are done by STEP B-2. 

Hence, we may assume $K_{X_i'} + \Delta' _i (c_i ') \not \equiv 0$, and 
replace $(X_i, \Delta_i (t))$ by $(X_i ', \Delta_i' (t))$ and continue the MMP. 
The MMP must terminate with a Mori fiber space $f_i : X_i \to Z_i$. 
If $\dim Z_i =0$, then the Picard number of $X_i$ is one. 
Therefore $(K_{X_i} + \Delta _i (c_i '))|_{T_i} \not \equiv 0$ for any component 
$T_i \subset \Supp \lfloor A_i \rfloor$, and we are done by STEP B-2. 
Suppose $\dim Z_i > 0$. 
Since $f_i$ is $\lfloor A_i \rfloor$-positive, 
$\lfloor A_i \rfloor$ dominates $Z_i$ and we are done by STEP B-3.


\vspace{2mm}

\noindent
\underline{\bf{STEP B-5}}\ \ 
In what follows, we assume that $(X_i, \Delta (c_i))$ is klt. 
We reduce to the case when $X_i$ has Picard number one. 

We run a $(K_{X_i} + B _i (c_i))$-MMP. 
Since $(K_{X_i} + B _i (c_i)) \equiv - A_i$ is not pseudo-effective, 
a $(K_{X_i} + B _i (c_i))$-MMP terminates and ends with a Mori fiber space 
by \cite[Corollary 1.3.3]{BCHM}. 

Let $f_i : X_i \dashrightarrow X_i '$ be a step of the MMP. 
First suppose that $f_i$ is birational. 
We write 
\[
A_i' := f_{i*} A_i, \quad B_i ' (t) := f_{i*} B_i (t), \quad  \Delta' _i (t) = A_i' + B_i ' (t). 
\]
Then, $(X_i ', \Delta' _i (t))$ satisfies all conditions in Proposition \ref{prop:accum_A} except for 
$K_{X_i'} + \Delta' _i (c_i ') \not \equiv 0$. 
We prove $K_{X_i'} + \Delta' _i (c_i ') \not \equiv 0$.

Suppose $K_{X_i'} + \Delta' _i (c_i ') \equiv 0$. 
It implies that $K_{X_i'} + \Delta' _i (c) \equiv 0$. 
Note that $(X_i ', \Delta' _i (c))$ is lc by Claim \ref{claim:lc}. 
Further the coefficients of $B_i'(c)$ satisfies the DCC (Lemma \ref{lem:coeff} (4)), 
and the coefficient of $A_i$ are approaching $1$ and $\lfloor A_i \rfloor = 0$. 
It contradicts Theorem \ref{thm:num_triv}. 

Since $K_{X_i'} + \Delta' _i (c_i ') \not \equiv 0$, 
we may replace $(X_i, \Delta _i (t))$ by $(X_i ', \Delta' _i (t))$ and continue the MMP. 
Then the MMP must terminate and ends with a Mori fiber space $X_i \to Z_i$. 
If $\dim Z_i = 0$, the Picard number of $X_i$ is one. 
Suppose $\dim Z_i > 0$. 
Since $f_i$ is $A_i$-positive, $\Supp A_i$ dominates $Z_i$, and we are done by STEP B-3. 

\vspace{2mm}

\noindent
\underline{\bf{STEP B-6}}\ \ 
We finish the case B. 

\begin{claim}\label{claim:ample}
We may assume that $K_{X_i} + A_i + B_i (c)$ is not ample for any $i$. 
\end{claim}
\begin{proof}
Assume that $K_{X_i} + A_i + B_i (c)$ is ample. 
We may write 
\[
A_i + B_i (t) = M_i + t(N_i ^+ - N_i ^-), 
\]
where $N_i ^+ \ge 0$ and $N_i ^- \ge 0$ have no common components. 
Further we may write 
\[
N_i ^+ \equiv n_i ^+ H_i, \quad N_i ^- \equiv n_i ^- H_i
\]
with some ample divisor $H_i$ and $n_i^+, n_i^- \in \mbQ _{\ge 0}$.

First suppose $c_i > c$. 
Then $K_{X_i} + A_i + B_i (c) \equiv (c-c_i)(N_i ^+ - N_i ^-)$ is ample, 
and so $n_i ^+ < n_i ^-$. 
Then we have
\[
K_{X_i} + M_i + cN_i ^+ - \bigg(c_i - (c_i -c)\frac{n_i ^+}{n_i ^-} \bigg)N_i ^-
\equiv
K_{X_i} + M_i + c_i(N_i ^+ - N_i ^-) \equiv 0. 
\]
Here, we have $c < c_i - (c_i -c)\frac{n_i ^+}{n_i ^-} < c_i$, 
and so
\[
0
\le
M_i + cN_i ^+ - \big(c_i - (c_i -c)\frac{n_i ^+}{n_i ^-} \big)N_i ^-
\le M_i + c(N_i ^+ - N_i ^-). 
\] 
Since $(X_i, M_i + c(N_i ^+ - N_i ^-))$ is lc by Claim \ref{claim:lc}, 
the new pair 
$\big( X_i, M_i + cN_i ^+ - \big(c_i - (c_i -c)\frac{n_i ^+}{n_i ^-} \big)N_i ^- \big)$ is also lc, 
but it contradicts Lemma \ref{lem:coeff} (4) and Theorem \ref{thm:num_triv}. 

Suppose $c_i < c$. 
Then we have $n_i ^+ > n_i ^-$, and
\[
K_{X_i} + M_i + \bigg(c_i + (c - c_i)\frac{n_i ^-}{n_i ^+} \bigg)N_i ^+ - cN_i ^-
\equiv
K_{X_i} + M_i + c_i(N_i ^+ - N_i ^-) \equiv 0. 
\]
Here, we have $c_i < c_i + (c - c_i)\frac{n_i ^-}{n_i ^+} < c$, 
and 
\[
0
\le
M_i + \bigg(c_i + (c - c_i)\frac{n_i ^-}{n_i ^+} \bigg)N_i ^+ - cN_i ^-
\le M_i + c(N_i ^+ - N_i ^-). 
\]
Note that the first inequality follows from Lemma \ref{lem:coeff} (3). 
Since $(X_i, M_i + c(N_i ^+ - N_i ^-))$ is lc by Claim \ref{claim:lc}, 
the new pair 
$\big( X_i, M_i + \big(c_i + (c - c_i)\frac{n_i ^-}{n_i ^+} \big)N_i ^+ - cN_i ^- \big)$ is also lc, 
but it contradicts Lemma \ref{lem:coeff} (4) and Theorem \ref{thm:num_triv}. 
\end{proof}

First suppose that $(X_i, \lceil A_i \rceil + B _i (c_i))$ is not lc. 
Note that $(X_i, \lceil A_i \rceil + B _i (c))$ is lc 
by Claim \ref{claim:lc}. 
Set 
\[
d_i :=  
\begin{cases}
\sup \{ t \in [c,c_i) \mid \text{$(X_i, \lceil A_i \rceil + B _i (t))$ is lc} \} 
& \text{when $c < c_i$, } \\
\inf \{ t \in (c_i, c] \mid \text{$(X_i, \lceil A_i \rceil + B _i (t))$ is lc} \} 
& \text{when $c_i < c$.}
\end{cases}
\]
Then $d_i \in \mfL _{d}(I) \subset \mfG _{d-1}(I)$, 
and $\lim d_i = \lim c_i = c$. 
Therefore we are done by induction on $d$. 

Thus we may assume that $(X_i, \lceil A_i \rceil + B _i (c_i))$ is lc. 
Set $e_i, f_i \in \mbR$ as 
\[
K_{X_i} + \lceil A_i \rceil + B _i (e_i) \equiv 0, \quad 
K_{X_i} + f_i \lceil A_i \rceil + B _i (c) \equiv 0. 
\]
Since $B_i (c_i) - B_i (c)$ is ample (Claim \ref{claim:ample}) and 
$K_{X_i} + A_i + B _i (c_i) \equiv 0$, there are only two cases:
\begin{itemize}
\item $e_i \ge c > c_i$ or $e_i \le c < c_i$, or
\item $c \ge e_i \ge c_i$ or $c \le e_i \le c_i$.
\end{itemize}

First suppose that $e_i \ge c > c_i$ or $e_i \le c < c_i$. 
Then $K_{X_i} + \lceil A_i \rceil + B _i (c)$ is ample, and so $f_i < 1$. 
Further, Since $K_{X_i} + A_i + B _i (c)$ is not ample, 
and the coefficients of $A_i$ are approaching one, 
it follows that $\lim f_i = 1$. 
Therefore, the set of coefficients of $f_i \lceil A_i \rceil + B _i (c)$
satisfies the DCC (Lemma \ref{lem:coeff} (4)), which contradicts Theorem \ref{thm:num_triv}. 

Next suppose that $c \ge e_i \ge c_i$ or $c \le e_i \le c_i$. 
Thus, we have $\lim e_i = \lim c_i = c$. 
Suppose $c \ge e_i \ge c_i$ (the other case can be proved in the same way), 
we may assume that $e_i < e_{i+1}$ for all $i$ or $e_i = c$ for some $i$. 
In the former case, as the sequence $e_i$ is accumulating to $c$, 
we may replace $(X_i, \Delta _i(t))$ by 
$(X_i, \lceil A_i \rceil + B _i (t))$. 
Remark that $(X_i, \lceil A_i \rceil + B _i (e_i))$ is lc, because both
$(X_i, \lceil A_i \rceil + B _i (c_i))$ and $(X_i, \lceil A_i \rceil + B _i (c))$ are lc. 
Note that the Picard number of $X_i$ is one. 
Hence for any component of $S_i \subset \Supp \lceil A_i \rceil$, 
we have $(K_{X_i} + \lceil A_i \rceil + B _i (c_i ')) |_{S_i} \not \equiv 0$ for some $c_i'$. 
Therefore we are done by STEP B-2. 
In the latter case, $c = e_i \in \mfG_{d-1}(I)$ by adjunction.

\vspace{3mm}

\noindent
\underline{\bf{Case A}}\ \ We treat the case when $a_i$ is bounded away from zero. 

In this case, it follows that $A_i = 0$ and $(X_i, B _i (c_i))$ is klt. 

\vspace{2mm}

\noindent
\underline{\bf{STEP A-1}}\ \ 
We reduce to the case when $X_i$ has Picard number one. 

Since $K_{X_i} + B _i (c_i) \equiv 0$ and $K_{X_i} + B _i (c_i ') \not \equiv 0$ for some $c_i'$, 
we can take $\epsilon \in \mbR$ such that $K_{X_i} + B _i (c_i + \epsilon)$ is klt 
(Lemma \ref{lem:coeff} (2)) 
and not pseudo-effective. 
We run a $(K_{X_i} + B _i (c_i + \epsilon))$-MMP. 
As $K_{X_i} + B _i (c_i + \epsilon)$ is not pseudo-effective, 
a $(K_{X_i} + B _i (c_i + \epsilon))$-MMP terminates and ends with 
a Mori fiber space \cite[Corollary 1.3.3]{BCHM}. 

Let $f_i : X_i \dashrightarrow X_i '$ be a step of the MMP. 
First suppose that $f_i$ is birational. 
We write 
\[
B_i ' (t) := f_{i*} B_i (t), \quad \Delta_i ' (t) := B_i ' (t). 
\]
Then, $(X_i ', \Delta' _i (t))$ satisfies all conditions in Proposition \ref{prop:accum_A} except for 
$K_{X_i'} + \Delta' _i (c_i ') \not \equiv 0$. 
We prove $K_{X_i'} + \Delta' _i (c_i ') \not \equiv 0$. 

Suppose $K_{X_i'} + \Delta' _i (c_i + \epsilon) \equiv 0$ (hence, $f_i$ is a divisorial contraction). 
We denote $D := K_{X_i} + \Delta _i (c_i + \epsilon)$, then we may write
\[
D \equiv D - f_i ^* f_{i*} D = a E, 
\]
where $E$ is the exceptional divisor and $a \in \mbR$. 
Since $D$ is not pseudo-effective, it follows that $a<0$. 
It contradicts the fact that $f_i$ is $D$-negative. 

Since $K_{X_i'} + \Delta' _i (c_i ') \not \equiv 0$, 
we may replace $(X_i, \Delta _i (t))$ by $(X_i ', \Delta _i ' (t))$, and 
continue the MMP. 
The MMP must terminate with a Mori fiber space $f_i : X_i \to Z_i$. 
If $\dim Z_i = 0$, then the Picard number of $X_i$ is one. 
Suppose $\dim Z_i > 0$. 
Let $F_i$ be the general fiber of $f_i$. Set $\Delta _{F_i}(t)$ by adjunction:
\[
(K_{X_i} + \Delta _i (t))|_{F_i} = K_{F_i} + \Delta _{F_i}(t). 
\]
Since $f_i$ is $(K_{X_i} + \Delta _i (c_i + \epsilon))$-negative, 
$(K_{X_i} + \Delta _i (c_i + \epsilon))|_{F_i} \not \equiv 0$. 
Then $(F_i, \Delta _{F_i}(t))$ satisfies the conditions in Proposition \ref{prop:accum_A}. 
Since $\dim F_i \le d-1$, we are done by induction on $d$. 

\vspace{2mm}

\noindent
\underline{\bf{STEP A-1'}}\ \ 
Since $X_i$ has Picard number one, 
by Lemma \ref{lem:picard1} and Lemma \ref{lem:coeff} (2), 
the number of components of $B_i(t)$ are bounded. 
Hence, possibly passing to a subsequence, we may assume that 
the number of components of $B_i(t)$ are fixed. 
Since $a_i$ are bounded away from zero, the coefficients of $B_i(t)$ 
have only finitely many possibilities (Lemma \ref{lem:coeff} (1)). 
Therefore, possibly passing to a subsequence, 
we may assume that the coefficients of $B_i (t)$ are fixed and of the form
\[
\frac{m-1+f+kt}{m}, 
\]
where $m \in \mbZ _{>0}$, $f \in I_+$, and $k \in \mbZ$. 
Here, $m$, $f$, and $k$ depend on the component but not on $i$. 

Set 
\begin{align*}
h_i ^+ &:= \sup \{ t \ge c_i \mid \text{$(X_i, B_i(t))$ is lc} \} \ge c_i, \\
h_i ^- &:= \inf \{ t \le c_i \mid \text{$(X_i, B_i(t))$ is lc} \} \le c_i. 
\end{align*}
Since $h_i ^+$ and $h_i ^-$ are bounded, possibly passing to a subsequence, 
we may assume that the limits $h^+ = \lim h_i ^+$, $h^- = \lim h_i ^-$ exist. 

\vspace{2mm}

\noindent
\underline{\bf{STEP A-2}}\ \ 
We finish the case when $c \ge h^+$ or $c \le h^-$. 

In this case, we have $h^+ = \lim h_i ^+ = c$ or $h^- = \lim h_i ^- = c$. 
Since $h_i ^+, h_i ^- \in \mfL _{d}(I) \subset \mfG _{d-1} (I)$, 
we are done by induction on $d$. 

\vspace{2mm}

\noindent
\underline{\bf{STEP A-3}}\ \ 
In what follows, we assume that $c < h^+$ and $c > h^-$. 
Let 
\[
d_i^+ = \frac{c_i + h_i^+}{2},\ \ 
d_i^- = \frac{c_i + h_i^-}{2},\ \ 
d^+ = \frac{c+h^+}{2}, \ \ d^- = \frac{c+h^-}{2}. 
\]
Then $d^+ > c$ and $d^- < c$. 
Further, we may assume $d^+ > c_i$ and $d^- < c_i$ possibly passing to a tail of the sequence. 
Note that the following hold: 
\begin{itemize}
\item If $c_i > c$, then $K_{X_i} + B_i (d^+)$ is ample. 
\item If $c_i < c$, then $K_{X_i} + B_i (d^-)$ is ample. 
\end{itemize}
This is because, $K_{X_i} + B_i (c)$ is not ample by the same reason as Claim \ref{claim:ample}. 

In this step, we prove that the following hold:
\begin{itemize}
\item If $c_i > c$, then $\vol(X_i, K_{X_i} + B_i (d^+))$ is unbounded. 
\item If $c_i < c$, then $\vol(X_i, K_{X_i} + B_i (d^-))$ is unbounded. 
\end{itemize}

Suppose that $c_i > c$ and 
$\vol (X_i, K_{X_i} + B_i (d^+))$ is bounded from above 
(the other case can be proved in the same way). 
Since the coefficients of $(X_i, B_i (d^+))$ are fixed, 
there exists $m \in \mbZ _{>0}$ such that 
$\phi _{m(K_{X_i} + B_i(d^+))}$ is birational for all $i$
by \cite[Theorem 1.3]{HMX2}. 
But then, by \cite[Lemma 2.4.2]{HMX}, 
$\{(X_i, B_i (d^+)) \mid i \in \mbZ _{>0}\}$ is log birationally bounded 
since $\vol(X_i, K_{X_i} + B_i (d^+))$ is bounded by the assumption. 
Note that the coefficients of $B_i (d^+ _i)$ are bounded from below and 
$\mld (X_i, B_i (d^+ _i))$ is also bounded from below:
\[
\mld (X_i, B_i (d^+ _i)) \ge \frac{\mld (X_i, B_i (h^+ _i)) + \mld (X_i, B_i (c_i))}{2}
= \frac{a_i}{2}. 
\]
Hence by \cite[Theorem 1.6]{HMX2}, 
$\{(X_i, B_i (d^+)) \mid i \in \mbZ _{>0}\}$ turns out to be a bounded family. 

Thus, we may take an ample Cartier divisor $H_i$ on $X_i$ such that 
\[
T_i \cdot H_i ^{\dim X_i -1}, \quad K_{X_i} \cdot H_i ^{\dim X_i -1}
\]
are bounded, where $T_i$ is any component of $B_i (t)$. 
Hence we may assume that these intersection numbers are independent of $i$ 
possibly passing to a subsequence. 
We may write $B_i (t) = M_i + t N_i$. 
As the coefficients of $B_i$ are independent of $i$, it follows that 
$M_i \cdot H_i^{\dim X_i -1}$ and $N_i \cdot H_i^{\dim X_i -1}$ are also constant. 
Since 
\[
0 = (K_{X_i} + B_i (c_i)) \cdot H_i ^{\dim X_i -1} = (K_{X_i} + M_i + c_i N_i) \cdot H_i ^{\dim X_i -1}, 
\]
it follows that $c_i$ is also constant, a contradiction. 
Remark that $N_i \cdot H_i ^{\dim X_i -1} \not = 0$ holds since $N_i \not \equiv 0$. 

\vspace{2mm}

\noindent
\underline{\bf{STEP A-4}}\ \ 
By STEP A-3, the following hold: 
\begin{itemize}
\item If $c_i > c$, then $K_{X_i} + B_i (d^+)$ is ample and $\vol(X_i, K_{X_i} + B_i (d^+))$ is unbounded.
\item If $c_i < c$, then $K_{X_i} + B_i (d^-)$ is ample and $\vol(X_i, K_{X_i} + B_i (d^-))$ is unbounded. 
\end{itemize}

Suppose $c_i > c$
(the other case can be proved in the same way). 
Note that $K_{X_i} + B_i (d^+) \equiv B_i (d^+) - B_i (c_i)$. 
Then, by Lemma 3.2.2 and Lemma 3.2.3 in \cite{HMX2}, 
possibly passing to a tail of the sequence, 
we may find $g_i < c_i$ and an $\mbR$-divisor $\Theta _i$ with the following conditions:
\begin{itemize}
\item $0 \le \Theta _i \sim _{\mbR} B_i(c_i) - B_i(g_i)$. 
\item $B_i (g_i) \ge 0$ (cf.\ Lemma \ref{lem:coeff} (3)). 
\item $\lim g_i = c$. 
\item $(X_i, B_i(g_i) + \Theta _i)$ has a unique non-klt place. 
\end{itemize}
Let $\phi : Y_i \to X_i$ be a dlt modification of $(X_i, B_i(g_i) + \Theta _i)$. 
Then we may write 
\[
K_{Y_i} + B_i' (g_i) + \Theta _i ' + S_i = \phi ^{*} (K_{X_i} + B_i(g_i) + \Theta _i), 
\]
where $S_i$ is the unique exceptional divisor, and 
$B_i' (t)$ and $\Theta _i '$ are the strict transform of $B_i (t)$ and $\Theta _i$. 
We may also write
\[
K_{Y_i} + B_i' (c_i) + s_i S_i =  \phi ^* (K_{X_i} + B_i (c_i))
\]
with $s_i < 1$ as $(X_i, B_i (c_i))$ is klt. 

\begin{claim}\label{claim:proportional}
We may assume that $S_i$ is ample and $K_{Y_i} + B_i' (l_i) + S_i \equiv 0$ for some $l_i \in [g_i, c_i)$. 
\end{claim}
First we assume this claim and finish the proof. 

Suppose that $(Y_i, B_i' (c_i) + S_i)$ is not lc. 
Note that $(Y_i, B_i' (g_i) + S_i)$ is lc. 
Set 
\[
k_i := \sup \{ t \in [g_i, c_i) \mid \text{$(Y_i, B_i' (t) + S_i)$ is lc}\}. 
\]
Then $k_i \in \mfL _d (I) \subset \mfG _{d-1}(I)$, and $\lim k_i = c$. 
Therefore we are done by induction on $d$. 

Thus, we may assume that $(Y_i, B_i' (c_i) + S_i)$ is lc. 
By adjunction, we can define $B_i '' (t)$ as follows:
\[
(K_{Y_i} + B_i' (t) + S_i) |_{S_i} = K_{S_i} + B_i '' (t). 
\]
Since $(Y_i, B_i' (c_i) + S_i)$ is lc, it follows that 
$B_i '' (t) \in \mcD _{c_i} (I)$ by Lemma \ref{lem:adjunction}. 
Further $(S_i, B_i'' (c_i))$ and $(S_i, B_i'' (g_i))$ are lc. 
By Claim \ref{claim:proportional}, it follows that 
$K_{S_i} + B_i'' (l_i) \equiv 0$ and $K_{S_i} + B_i'' (c_i) \not \equiv 0$. 
Therefore $l_i \in \mfG _{d-1} (I)$. 
Since $\lim l_i = c$, we are done by induction on $d$. 

\begin{proof}[Proof of Claim \ref{claim:proportional}]
We run a $(K_{Y_i} + B_i' (g_i) + \Theta _i ')$-MMP. 
Since $(Y_i, B_i' (g_i) + \Theta _i ')$ is klt and 
$K_{Y_i} + B_i' (g_i) + \Theta _i ' \equiv - S_i$ is not pseudo-effective, 
a $(K_{Y_i} + B_i' (g_i) + \Theta _i ')$-MMP $f_i : Y_i \dashrightarrow W_i$ terminates and ends with a Mori fiber space 
$\pi _i : W_i \to Z_i$
by \cite[Corollary 1.3.3]{BCHM}. 

Let $F_i$ be the general fiber of $\pi _i$ and 
let $B_i'''(t)$, $\Theta _i '''$ and $S_i '''$ be the restriction of 
$f_{i*}B_i'(t)$, $f_{i*}\Theta _i '$ and $f_{i*}S_i$ to $F_i$. 
Note that $S_i ''' \not = 0$ since every step of this MMP is $S_i$-positive. 
Further $B_i'''(t)$, $\Theta _i '''$ and $S_i '''$ are multiples of the same ample divisor. 
Therefore $S_i '''$ is ample. 
Since 
\[
K_{F_i} + B_i''' (g_i) + \Theta _i ''' + S_i''' \equiv 0, \ \text{and}\ \ \  
K_{F_i} + B_i''' (c_i) + s_i S_i''' \equiv 0, 
\]
we may find 
\[
K_{F_i} + B_i''' (l_i) + S_i ''' \equiv 0
\]
for some $l_i \in [g_i, c_i)$. 
Therefore, 
we can apply the same argument above after replacing $(Y_i, B_i'(t)+S_i)$ by $(F_i, B_i'''(t) + S_i ''')$. 
\end{proof}
\end{proof}

\begin{proof}[Proof of Corollary \ref{cor:rationality}]
Since $\mfG_{d} (I) \subset \Span _{\mbQ} (I \cup \{1\})$, 
the statement follows from Theorem \ref{thm:local_global} and Theorem \ref{thm:accum_mfG}. 
\end{proof}

\section{Perturbation of irrational coefficients of lc pairs}\label{section:perturbe}
The goal of this section is to prove Theorem \ref{thm:perturbe}. 
The ideal setting is treated as Theorem \ref{thm:perturbe2}. 

\begin{proof}[Proof of Theorem \ref{thm:perturbe}]
We may write the $\mbQ$-linear functions $s_i$ as 
\[
s_i (x_0, \ldots, x_{c'}) = \sum_{0 \le j \le c'} q_{ij}x_j
\]
with $q_{ij} \in \mbQ$. 
Since $s_i (r_0, \ldots , r_{c'}) \in \mbR _{\ge 0}$ and 
$r_0, \ldots , r_{c'}$ are $\mbQ$-linearly independent, 
we can take $t^-, t^+ \in \mbQ$ with the following conditions:
\begin{itemize}
\item $t^- < r_{c'} < t^+$, and 
\item $s_i (r_0, \ldots , r_{c'-1}, t) \in \mbR _{\ge 0}$ holds for any $t$ 
satisfying $t^- \le t \le t^+$. 
\end{itemize}

Suppose that the statement does not hold. 
Then there exist $\mbQ$-Gorenstein varieties $X^{(l)}\ (l \in \mbZ_{>0})$ of dimension $d$ and 
$\mbQ$-Cartier effective Weil divisors 
$D_{0}^{(l)}, \ldots , D_{c}^{(l)}$ on $X^{(l)}$ such that the following holds:
\begin{itemize}
\item $\big( X^{(l)}, \sum_{1 \le i \le c} s_i(r_0, \ldots , r_{c'}) D_i ^{(l)} \big)$ is lc, and
\item $\lim h^+ _l = r_{c'}$ or $\lim h^- _l = r_{c'}$, 
\end{itemize}
where we set
\begin{align*}
h^+ _l &:= \sup \big\{ t \ge r_{c'} \ \big| \  \text{$ \big(X^{(l)}, 
\sum_{1 \le i \le c} s_i(r_0, \ldots, r_{c'-1}, t) D_i ^{(l)} \big)$ is lc} \big\}, \\
h^- _l &:= \inf \big\{ t \le r_{c'} \ \big| \  \text{$(X^{(l)}, 
\sum_{1 \le i \le c} s_i(r_0, \ldots, r_{c'-1}, t) D_i ^{(l)} \big)$ is lc} \big\}. 
\end{align*}

Suppose that $\lim h^- _l = r_{c'}$ (the other case can be proved in the same way). 
We may assume that $t^- \le h^- _l \le r_{c'}$. 
Note that 
\begin{align*}
\sum_{1 \le i \le c} s_i(r_0, \ldots, r_{c'-1}, t) D_i ^{(l)}
=  \sum_{1 \le i \le c} s_i(r_0, \ldots, &r_{c'-1}, t^-) D_i ^{(l)} \\
&+ (t - t^-) \sum_{1 \le i \le c} q_{ic'}D_{i} ^{(l)}. 
\end{align*}
Let 
\[
I := \{ s_i(r_0, \ldots, r_{c'-1}, t^-) \mid 1 \le i \le c \}. 
\]
This becomes a finite set. 
Take $m \in \mbZ _{>0}$ such that $mq_{ic'} \in \mbZ$ holds for any $i$. 
Then $\frac{h_l ^- - t^-}{m} \in \mfL _{d} (I)$. 
Hence, by Corollary \ref{cor:rationality}, 
it follows that 
\[
\frac{r_{c'} - t^-}{m}\in \Span _{\mbQ} (I \cup \{ 1 \}) \subset \Span _{\mbQ} (r_0, \ldots, r_{c' -1}). 
\]
It contradicts the $\mbQ$-linearly independence of $r_0, \ldots , r_{c'}$. 
\end{proof}

The case of the pair with ideal sheaves can be also proved. 
\begin{thm}\label{thm:perturbe2}
Fix $d \in \mbZ _{>0}$. Let $r_1, \ldots, r_{c'}$ be positive real numbers and let $r_0 = 1$. 
Assume that $r_0, \ldots, r_{c'}$ are $\mbQ$-linearly independent. 
Let $s_1, \ldots, s_c : \mbR^{c'+1} \to \mbR$ be $\mbQ$-linear functions from $\mbR ^{c'+1}$ to $\mbR$. 
Assume that $s_i (r_0, \ldots , r_{c'}) \in \mbR _{\ge 0}$ for each $i$. 
Then there exists a positive real number $\epsilon >0$ with the following conditions:
\begin{itemize}
\item $s_i(r_0, \ldots, r_{c' -1}, t) \ge 0$ holds for any $t$ satisfying $|t - r_{c'}| \le \epsilon$. 
\item For any $\mbQ$-Gorenstein normal variety $X$ of dimension $d$ and any 
ideal sheaves $\mfa_1, \ldots, \mfa_c$ on $X$, 
if $(X, \prod _{1 \le i \le c} \mfa_i ^{s_i(r_0, \ldots, r_{c'}) })$ is lc, then 
$(X, \prod _{1 \le i \le c}  \mfa_i ^{s_i(r_0, \ldots, r_{c' -1}, t)})$ is also lc for any $t$ 
satisfying $|t - r_{c'}| \le \epsilon$. 
\end{itemize}
\end{thm}

This theorem follows from Theorem \ref{thm:perturbe} by the following lemma (cf.\ \cite[Proposition 9.2.28]{Laz2}). 

\begin{lem}\label{lem:ideal}
Fix $l \in \mbZ _{>0}$. 
Let $X$ be a $\mbQ$-Gorenstein normal affine variety, and 
let $\mfa _1 , \ldots , \mfa _c$ be ideal sheaves on $X$. 
Fix general elements 
$
f_{i1}, \ldots , f_{il} \in \mfa _i
$
for each $i$, and let $D_{ij} = \divi (f_{ij}) \ge 0$ be the corresponding Cartier divisors. 
Set $D_i := \sum _{1 \le j \le l} D_{ij}$. 

Then the following holds for any positive real numbers $r_1, \ldots, r_c \le l$ at most $l$: 
the pair $(X, \prod _{1 \le i \le c} \mfa _i ^{r_i})$ is lc if and only if 
the pair $(X, \frac{1}{l}\sum _{1 \le i \le c} r_i D _i )$ is lc. 
\end{lem}
\begin{defi}\label{def:general}
Let $X$ be an affine variety and $\mfa$ an ideal sheaf. 
Fix generators $g_1, \ldots , g_c \in \mfa$. 
Then, a \textit{general element} of $\mfa$ is a general $\mbC$-linear combination of $g_i$. 
\end{defi}
\begin{proof}[Proof of Lemma \ref{lem:ideal}]
Let $\mfb _i$ be an ideal sheaf generated by $\prod _{1 \le j \le l} f_{ij}$. 
Then the pair $(X, \frac{1}{l}\sum _{1 \le i \le c} r_i D _i )$ is corresponding to 
the pair $(X, \prod _{1 \le i \le c} \mfb _i ^{r_i/l})$.

Since $\mfb_i \subset \mfa _i ^l$, it easily follows that the log canonicity of 
$(X, \prod _{1 \le i \le c} \mfb _i ^{r_i/l})$ implies 
the log canonicity of $(X, \prod _{1 \le i \le c} \mfa _i ^{r_i})$.

Suppose that $(X, \prod _{1 \le i \le c} \mfa _i ^{r_i})$ is lc. 
Let $Y \to X$ be a log resolution of $(X, \prod _{1 \le i \le c} \mfa _i ^{r_i})$. 
Then we may write $\mfa _i \mcO _Y = \mcO_Y (- E_i)$ with some Cartier divisor $E_i$. 
Since $\mfb_i \subset \mfa _i ^l$, we may write 
$\mfb_i \mcO _Y = \mfc_i \mcO _Y (- l E_i)$
with some ideal sheaf $\mfc_i \subset \mcO _Y$. 
Let $e_i$ be a local generator of $\mcO _Y (- E_i)$. 
Then $\mfc_i$ is generated by $\prod _{1 \le j \le l} g_{ij}$, 
where we set $g_{ij} := f_{ij} e_{i} ^{-1} \in \mcO _Y$. 
As $f_{i1}, \ldots, f_{il}$ are general elements of $\mfa _i$, 
the elements $g_{i1}, \ldots , g_{il}$ become general elements of $\mcO _Y$. 
Therefore $Y \to X$ is also a log resolution of $(X, \prod _{1 \le i \le c} \mfb _i ^{r_i/l})$. 
Since $\ord _{g_{ij}} \mfb _i^{r_i /l} = r_i /l \le 1$, 
it follows that $(X, \prod _{1 \le i \le c} \mfb _i ^{r_i/l})$ is also lc. 
\end{proof}

\section{Proof of main theorem and corollaries}\label{section:main}
Theorem \ref{thm:main} can be proved by the induction on $\dim _{\mbQ} \Span_{\mbQ} (I \cup \{ 1 \})$. 
The same argument essentially appears in \cite{Kawakita:discrete}. 
\begin{proof}[Proof of Theorem \ref{thm:main}]
It is sufficient to prove the case when $1 \in I$. 
Let $r_0 = 1, r_1 , \ldots , r_c$ be all the elements of $I$. 
Set $c' + 1  := \dim _{\mbQ} \Span_{\mbQ} (1, r_1, \ldots , r_c)$. 
Possibly rearranging the indices, 
we may assume that $r_0 , \ldots , r_{c'}$ are $\mbQ$-linearly independent. 
We may write $r_i = \sum_{0 \le j \le c'} q_{ij}r_j$ with $q_{ij} \in \mbQ$. 

We prove by induction on $c'$. 
If $c' = 0$, we can take $n \in \mbZ _{>0}$ such that 
$I \subset \frac{1}{n} \mbZ$ and $\frac{1}{r} \in \frac{1}{n} \mbZ$. 
Then $B(d,r,I) \subset \frac{1}{n} \mbZ$ and $B(d,r,I)$ turns out to be discrete. 

Set $\mbQ$-linear functions $s_0, \ldots , s_c$ as follows: 
\[s_i : \mbR ^{c'+1} \to \mbR; \quad s_i (x_0, \ldots , x_{c'}) = \sum _{0 \le j \le c'} q_{ij}x_j. \]
Take $\epsilon > 0$ as in Theorem \ref{thm:perturbe2}. 
We fix $t^+, t^- \in \mbQ$ such that 
\[
t^+ \in (r_{c'}, r_{c'} + \epsilon] \cap \mbQ, \quad 
t^- \in [r_{c'} - \epsilon, r_{c'}) \cap \mbQ. 
\]
We define $r^+ _0, \ldots, r^+ _c$ and $r^- _0, \ldots, r^- _c$ as
\[
r_i ^+ = s_i (r_0, \ldots, r_{c' -1}, t^+), \quad 
r_i ^- = s_i (r_0, \ldots, r_{c' -1}, t^-). 
\]
Further, we set $I' := \{ r^+ _0, \ldots , r^+ _c, r^- _0, \ldots , r^- _c \}$. 
Then $\dim _{\mbQ} \Span _{\mbQ} (I') = c'$, 
and so $B(d,r,I')$ is discrete by induction.

Let $(X, \prod _{0 \le i \le c} \mfa _i ^{r_i} ) \in P(d,r)$, and let $E$ be a divisor over $X$. 
Since $(X, \prod _{0 \le i \le c}  \mfa_i ^{r _i})$ is lc, 
$(X, \prod _{0 \le i \le c}  \mfa_i ^{r^* _i})$ is also lc for each $* \in \{ +, - \}$. 
Hence we have 
\begin{align*}
0 &\le a_E (X, \prod _{0 \le i \le c}  \mfa_i ^{r^* _i}) \\ 
&= a_E (X, \prod _{0 \le i \le c}  \mfa_i ^{r_i}) - 
(r^* _{c'} - r_{c'}) \sum _{0 \le i \le c} q_{ic'}\ord _E \mfa_i. 
\end{align*}
Therefore, either of the following holds:
\begin{itemize}
\item $0 \le \sum _{0 \le i \le c} q_{ic'}\ord _E \mfa_i \le 
\epsilon _+ ^{-1} a_E (X, \prod _{0 \le i \le c} \mfa_i ^{r_i}) $, or
\item $- \epsilon _- ^{-1} a_E (X, \prod _{0 \le i \le c} \mfa_i ^{r_i}) \le 
\sum _{0 \le i \le c} q_{ic'}\ord _E \mfa_i \le 0$, 
\end{itemize}
where we set $\epsilon _+ := r^+ _{c'} - r_{c'}$ and $\epsilon _- := r_{c'} - r^- _{c'}$.

It is sufficient to show the discreteness of $B(d,r,I) \cap [0,a]$ for any $a \in \mbR _{>0}$. 
Take $n \in \mbZ _{>0}$ such that $q_{ic'} \in \frac{1}{n} \mbZ$ holds for any $i$. 
Then, it is sufficient to prove that $B(d,r,I) \cap [0,a]$ is contained in 
\begin{align*}
\big\{ b + \epsilon_+ e  \ & \big | \  b \in B(d,r,I'), 
e \in \frac{1}{n}\mbZ \cap[0, \epsilon _+ ^{-1} a]  \big\} \\
&\cup 
\big\{ b - \epsilon_- e  \ \big | \  b \in B(d,r,I'), 
e \in \frac{1}{n}\mbZ \cap[- \epsilon _- ^{-1} a, 0]  \big\}. 
\end{align*}
In fact, this set becomes discrete because $B(d,r,I')$ is discrete, and both
$\frac{1}{n}\mbZ \cap[0, \epsilon _+ ^{-1} a]$ and $\frac{1}{n}\mbZ \cap[- \epsilon _- ^{-1} a, 0]$ 
are finite. 

Let $(X, \prod _{0 \le i \le c} \mfa_i ^{r_i}) \in P(d,t)$, and $E$ a divisor over $X$. 
Assume that $a_E(X, \prod _{0 \le i \le c} \mfa_i ^{r_i}) \in [0,a]$ holds. 
Further, suppose $\sum _{0 \le i \le c} q_{ic'}\ord _E \mfa_i \ge 0$ 
(the other case can be proved in the same way). 
Then, we have 
\begin{align*}
a_E (X, \prod _{0 \le i \le c} \mfa_i ^{r_i})
=
a_E (X, \prod _{0 \le i \le c} \mfa_i ^{r^+ _i})
+
(r^+ _{c'} -r_{c'})\sum _{0 \le i \le c} q_{ic'}\ord _E \mfa_i. 
\end{align*}
Here, we have 
\begin{itemize}
\item $a_E (X, \prod _{0 \le i \le c} \mfa_i ^{r^+ _i}) \in B(d,r,I')$, 
\item $r^+ _{c'} -r_{c'} = \epsilon _+$, and 
\item $\sum _{0 \le i \le c} q_{ic'}\ord _E \mfa_i \in \frac{1}{n}\mbZ 
\cap[0, \epsilon _+ ^{-1} a]$. 
\end{itemize}
We complete the proof. 
\end{proof}

\begin{proof}[Proof of Corollary \ref{cor:acc_3dimcan}]
Note that $A_{\text{can}}(3,I) \subset [1,3]$ holds (cf.\ \cite{Kawamata}, \cite{Mark}). 
We prove that for any $a > 1$, the set
\[
A_{\text{can}}(3,I) \cap [a, + \infty)
\]
is a finite set. 

By the classification of three-dimensional $\mbQ$-factorial terminal singularities 
(see \cite{Kawamata}, \cite{Mark}), 
the minimal log discrepancy of a three-dimensional terminal singularity is equal to $1+1/r\ (r \in \mbZ_{>0})$ or $3$. 
In the case when $\mld _x (X)=3$, the Gorenstein index of $X$ at $x$ is $1$. 
If $\mld _x (X)=1+1/r$, the Gorenstein index of $X$ at $x$ is $r$. 
Further, by \cite[Corollary 5.2]{Kawamata:crepant}, if $X$ has Gorenstein index $r$ at $x \in X$, then 
$rD$ is Cartier at $x$ for any Weil divisor $D$. 

Let $(X, \Delta)$ be a three-dimensional canonical pair satisfying $\Delta \in I$ and $\mld _x(X, \Delta) \ge a$. 
By \cite[Corollary 1.4.3]{BCHM}, 
there exists a projective morphism $f : Y \to X$ with the following properties: 
\begin{itemize}
\item $Y$ is a $\mbQ$-factorial terminal variety. 
\item $f^* (K_X + \Delta) = K_Y + \Delta _Y$ holds, where $\Delta _Y$ is the strict transform 
on $Y$ of $\Delta$ (note that $(X, \Delta)$ is canonical). 
\end{itemize}
Take a divisor $E$ over $X$ such that 
$\mld _x(X, \Delta) = a_E (X, \Delta)$ and $\cent _X (E) = \{x \}$.

Suppose $\dim \cent _Y (E) = 0$. 
Then $\mld _x(X, \Delta) = \mld _y (Y, \Delta _Y)$ holds, where $\{ y \} := \cent _Y (E)$. 
Since $\mld _y (Y) \ge \mld _y (Y, \Delta _Y) \ge a$ holds, 
the Gorenstein index of $Y$ at $y$ is at most $\lfloor \frac{1}{a-1} \rfloor$. 
Let $l$ be the Gorenstein index of $Y$ at $y$. 
Since $lD$ is Cartier at $y$ for any Weil divisor $D$ on $Y$, 
it follows that $\mld _y (Y, \Delta _Y) \in A'(3,l, \frac{1}{l}I)$ (see Remark \ref{rmk:cartier}), where 
we set 
\[
\frac{1}{l}I := \{ fl^{-1} \mid f \in I \}. 
\]
Therefore we have 
\[
\mld _x (X, \Delta) \in \bigcup _{l \le \lfloor \frac{1}{a-1} \rfloor} A' \big( 3,l, \frac{1}{l}I \big), 
\]
and the right hand side is a finite set by Corollary \ref{cor:main}.

Suppose $\dim \cent _Y (E) = 1$. 
Then, by \cite[Proposition 2.1]{Ambro:mld}, 
\[
\mld _y (Y, \Delta _Y) = 1 + \mld _x(X, \Delta)
\]
holds for some $y \in \cent _Y (E)$. 
Since $\mld _y (Y) \ge 1 + a > 2$, it follows that $Y$ has Gorenstein index $1$. 
Hence, 
\[
\mld _y (Y, \Delta _Y) \in A'(3,1,I). 
\]
Therefore, we have
\[
\mld _x (X, \Delta) \in -1 + A'(3,1,I), 
\]
and the right hand side is a finite set by Corollary \ref{cor:main}. 

Suppose $\dim \cent _Y (E) = 2$. 
Then $E$ is a divisor on $Y$, and we have
\[
\mld _x (X, \Delta) = 1 - \coeff _E \Delta _Y. 
\]
Therefore, we have 
\[
\mld _x (X, \Delta) \in 1 - I, 
\]
and the right hand side is a finite set. 
\end{proof}

\section*{Acknowledgments}
The author expresses his gratitude to his advisor Professor Yujiro Kawamata 
for his encouragement and valuable advice. 
He is grateful to 
Professors Yoshinori Gongyo, Masayuki Kawakita, Shunsuke Takagi, and Hiromu Tanaka for useful comments and suggestions.
He is supported by the Grant-in-Aid for Scientific Research
(KAKENHI No.\ 25-3003) 
and the Program for Leading Graduate Schools, MEXT, Japan.


\begin{bibdiv}
 \begin{biblist*}

\bib{Alexeev}{article}{
   author={Alexeev, Valery},
   title={Two two-dimensional terminations},
   journal={Duke Math. J.},
   volume={69},
   date={1993},
   number={3},
   pages={527--545},
}





\bib{Ambro:mld}{article}{
   author={Ambro, Florin},
   title={On minimal log discrepancies},
   journal={Math. Res. Lett.},
   volume={6},
   date={1999},
   number={5-6},
   pages={573--580},
}


\bib{Ambro:toric}{article}{
   author={Ambro, Florin},
   title={The set of toric minimal log discrepancies},
   journal={Cent. Eur. J. Math.},
   volume={4},
   date={2006},
   number={3},
   pages={358--370 (electronic)},
}


\bib{BCHM}{article}{
   author={Birkar, Caucher},
   author={Cascini, Paolo},
   author={Hacon, Christopher D.},
   author={McKernan, James},
   title={Existence of minimal models for varieties of log general type},
   journal={J. Amer. Math. Soc.},
   volume={23},
   date={2010},
   number={2},
   pages={405--468},
}


\bib{dFM:limit}{article}{
   author={de Fernex, Tommaso},
   author={Musta{\c{t}}{\u{a}}, Mircea},
   title={Limits of log canonical thresholds},
   journal={Ann. Sci. \'Ec. Norm. Sup\'er. (4)},
   volume={42},
   date={2009},
   number={3},
   pages={491--515},
}





\bib{Fujino:fundamental}{article}{
   author={Fujino, Osamu},
   title={Fundamental theorems for the log minimal model program},
   journal={Publ. Res. Inst. Math. Sci.},
   volume={47},
   date={2011},
   number={3},
   pages={727--789},
}

\bib{HMX}{article}{
   author={Hacon, Christopher D.},
   author={McKernan, James},
   author={Xu, Chenyang},
   title={On the birational automorphisms of varieties of general type},
   journal={Ann. of Math. (2)},
   volume={177},
   date={2013},
   number={3},
   pages={1077--1111},
}



\bib{HMX2}{article}{
   author={Hacon, Christopher D.},
   author={McKernan, James},
   author={Xu, Chenyang},
   title={ACC for log canonical thresholds},
   journal={Ann. of Math. (2)},
   volume={180},
   date={2014},
   number={2},
   pages={523--571},
}


\bib{Kawakita:BDD}{article}{
   author={Kawakita, Masayuki},
   title={Towards boundedness of minimal log discrepancies by the
   Riemann-Roch theorem},
   journal={Amer. J. Math.},
   volume={133},
   date={2011},
   number={5},
   pages={1299--1311},
}


\bib{Kawakita:discrete}{article}{
   author={Kawakita, Masayuki},
   title={Discreteness of log discrepancies over log canonical triples on a fixed pair},
   eprint={arXiv:1204.5248v1}
}




\bib{Kawakita:connectedness}{article}{
   author={Kawakita, Masayuki},
   title={A connectedness theorem over the spectrum of a formal power series ring},
   eprint={arXiv:1403.7582v1}
}


\bib{Kawamata:crepant}{article}{
   author={Kawamata, Yujiro},
   title={Crepant blowing-up of $3$-dimensional canonical singularities and
   its application to degenerations of surfaces},
   journal={Ann. of Math. (2)},
   volume={127},
   date={1988},
   number={1},
   pages={93--163},
}



\bib{Kawamata}{article}{
   author={Kawamata, Yujiro},
   title={The minimal discrepancy coefficients of terminal singularities in dimension 3},
   journal={Appendix to V. V. Shokurov, 
   \textit{Three-dimensional log perestroikas}, Izv. Ross. Akad. Nauk Ser. Mat.},
   volume={56},
   date={1992},
   number={1},
   pages={105--203},
   issn={0373-2436},
}

\bib{Kollar:which}{article}{
   author={Koll\'ar, J\'anos},
   title={Which powers of holomorphic functions are integrable?},
   eprint={arXiv:0805.0756v1}
}

\bib{KM}{book}{
   author={Koll{\'a}r, J{\'a}nos},
   author={Mori, Shigefumi},
   title={Birational geometry of algebraic varieties},
   series={Cambridge Tracts in Mathematics},
   volume={134},
   publisher={Cambridge University Press},
   date={1998},
}

\bib{Kollars}{collection}{
   author={Koll{\'a}r, J{\'a}nos, \ et al.}, 
   title={Flips and abundance for algebraic threefolds},
   note={Papers from the Second Summer Seminar on Algebraic Geometry held at
   the University of Utah, Salt Lake City, Utah, August 1991;
   Ast\'erisque No. 211 (1992)},
   publisher={Soci\'et\'e Math\'ematique de France, Paris},
   date={1992},
   pages={1--258},
}


\bib{Laz2}{book}{
   author={Lazarsfeld, Robert},
   title={Positivity in algebraic geometry. II},
   series={Ergebnisse der Mathematik und ihrer Grenzgebiete. 3. Folge. A
   Series of Modern Surveys in Mathematics 
   },
   volume={49},
   publisher={Springer-Verlag, Berlin},
   date={2004},
   pages={xviii+385},
   isbn={3-540-22534-X},
}


\bib{Mark}{article}{
   author={Markushevich, Dimitri},
   title={Minimal discrepancy for a terminal cDV singularity is $1$},
   journal={J. Math. Sci. Univ. Tokyo},
   volume={3},
   date={1996},
   number={2},
   pages={445--456},
}


\bib{MP}{article}{
   author={McKernan, James},
   author={Prokhorov, Yuri},
   title={Threefold thresholds},
   journal={Manuscripta Math.},
   volume={114},
   date={2004},
   number={3},
   pages={281--304},
}




\bib{Shokurov:acc}{article}{
   author={Shokurov, V. V.},
   title={A.c.c. in codimension 2},
   date={1993, preprint},
}


\bib{Shokurov:models}{article}{
   author={Shokurov, V. V.},
   title={$3$-fold log models},
   note={Algebraic geometry, 4},
   journal={J. Math. Sci.},
   volume={81},
   date={1996},
   number={3},
   pages={2667--2699},
}

\bib{Shokurov:letter}{article}{
   author={Shokurov, V. V.},
   title={Letters of a bi-rationalist. V. Minimal log discrepancies and
   termination of log flips},
   journal={Tr. Mat. Inst. Steklova},
   volume={246},
   date={2004},
   number={Algebr. Geom. Metody, Svyazi i Prilozh.},
   pages={328--351},
   translation={
      journal={Proc. Steklov Inst. Math.},
      date={2004},
      number={3 (246)},
      pages={315--336},
   },
}




\end{biblist*}
\end{bibdiv}
\end{document}